\documentclass[a4paper,10pt]{article}
\usepackage{a4wide}
\usepackage{amsmath}
\usepackage{amssymb}
\usepackage{amsthm}
\usepackage{latexsym}
\usepackage{graphicx}
\usepackage[english]{babel}
\usepackage{makeidx}

\newtheorem{theorem}{Theorem}[section]
\newtheorem{lemma}[theorem]{Lemma}
\newtheorem{corollary}[theorem]{Corollary}
\newtheorem{proposition}[theorem]{Proposition}

\theoremstyle{definition}
\newtheorem{definition}[theorem]{Definition}

\theoremstyle{remark}

\begin{document}
\title{A class of transversal polymatroids with Gorenstein base ring}
\author{Alin \c{S}tefan\\
 "Petroleum and Gas" University of Ploie\c sti, Romania}
\date{}
\maketitle

\begin{abstract}
In this paper, the principal tool to describe transversal polymatroids with Gorenstein base ring is polyhedral  geometry, especially the $Danilov-Stanley$ \ theorem for the characterization of canonical module. Also, we compute the $a-invariant$ and the Hilbert series of base ring associated to this class of transversal polymatroids.
\end{abstract}

\section{Introduction}
In this paper we determine the facets of the polyhedral cone generated by the exponent set of the monomials defining the base ring associated to a transversal polymatroid. The importance of knowing those facets comes from the fact that the canonical module of the base ring can be expressed in terms of the relative interior of the cone. This would allow to compute the $a-invariant$ of those base rings. The results presented were discovered by extensive computer algebra experiments
performed with {\it{Normaliz}} \ \cite{BK}.

The author would like to thank Professor Dorin Popescu for valuable suggestions and comments during the preparation of this paper.

\section{Preliminaries}
\[\]
Let $n\in \mathbb{N},$ \ $n \geq 3,$ $\sigma \in S_{n},$ \
$\sigma=(1,2,\ldots, n)$ \ the cycle of length $n,$ \ $[n]:=\{1,
2,\ldots, n\}$ \ and $\{e_{i}\}_{1 \leq i \leq n}$ \ be the
canonical base of $\mathbb{R}^{n}.$ \ For a vector $x\in \mathbb{R}^{n},$ $x=(x_{1},\ldots,x_{n})$ \ we will denote by $ | \ x \ |,$ \ $| \ x \ | := x_{1}+ \ldots + x_{n}.$ If $x^{a}$ is a monomial in $K[x_{1}, \ldots, x_{n}]$ we set $log(x^{a})=a$. Given a set $A$ of monomials, the $log \ set \ of \ A,$ denoted $log(A),$ consists of all  $log(x^{a})$ with $x^{a}\in A.$\\We consider the following set of integer vectors of $\mathbb{N}^{n}$:

\[\begin{array}{ccccccccccccccccc}
     &  &  &  &  &  &  &  &  &  &  &  &  &  &\downarrow i^{th} column &
  \end{array}
\]

$\nu_{\sigma^{0}[i] } := \left( \begin{array}{cccccccccc}
    -(n-i-1), & -(n-i-1), & \ldots & ,-(n-i-1), & (i+1), & \ldots & ,(i+1)  \\
   \end{array}
   \right),
$

\[\begin{array}{ccccccccccccccccccccccccccccccccccccccccccccccccccccc}
     &  &\downarrow (2)^{th} column &  &  &\downarrow (i+1)^{th} column &
  \end{array}
\]

$\nu_{\sigma^{1}[i] } := \left( \begin{array}{ccccccc}
     (i+1), & -(n-i-1), & \ldots & ,-(n-i-1), & (i+1), & \ldots & ,(i+1)  \\
   \end{array}
   \right),
$

\[\begin{array}{cccccccccccccccccccccc}
     &  &  &  &  &  &  &  &  &  &  &\downarrow (3)^{th} column  &  &  &\downarrow (i+2)^{th} column &
  \end{array}
\]

$\nu_{\sigma^{2}[i] } := \left( \begin{array}{cccccccc}
     (i+1), & (i+1)& ,-(n-i-1), & \ldots & ,-(n-i-1), & (i+1), & \ldots & ,(i+1)  \\
   \end{array}
   \right),
$
\[\ldots \ldots \ldots \ldots \ldots \ldots \ldots \ldots \ldots \ldots \ldots \ldots \ldots \ldots \ldots \ldots \ldots
\ldots \ldots \ldots \ldots \ldots \ldots \ldots \ldots \ldots
\ldots \ldots \ldots \ldots\]
\[\]
\[\begin{array}{ccccccccccccccccccccccc}
     &  &  &  &  &  &  &  &  &  &  &  &  &\downarrow (i-2)^{th} column  &  &  & \downarrow (n-2)^{th} column &
  \end{array}
\]

$\nu_{\sigma^{n-2}[i] } := \left(
\begin{array}{ccccccccccccccc}
     -(n-i-1), \ldots , -(n-i-1) , (i+1) , & \ldots & , (i+1) , -(n-i-1) , -(n-i-1)  \\
   \end{array}
   \right),
$

\[\begin{array}{ccccccccccccccccccccccccccccccc}
     &  &  &  &  &  &  &  &  &  &  &  &  &  &  &  &\downarrow (i-1)^{th} column  &  &  & \downarrow (n-1)^{th} column &
  \end{array}
\]

$\nu_{\sigma^{n-1}[i] } := \left( \begin{array}{cccccccc}
     -(n-i-1), & \ldots & , -(n-i-1) ,(i+1), & \ldots & ,(i+1) ,-(n-i-1)  \\
   \end{array}
   \right),\\
$

where $\sigma^{k}[i]:=\{\sigma^{k}(1), \ldots , \sigma^{k}(i)\}$
for all $1\leq i \leq n-1$ and $0\leq k \leq n-1.$\\

{\it{Remark}:} $\nu_{\sigma^{k}[n-1]}=n \ e_{[n]\setminus
\sigma^{k}[n-1]}$ for all $0\leq k
\leq n.$\\

For example, if $n=4,$ $\sigma =(1,2,3,4)\in S_{4}$ then we have
the following set of integer vectors:\\
\[\nu_{\sigma^{0}[1]}=\nu_{\{1\}}=(-2,2,2,2), \ \ \ \ \nu_{\sigma^{0}[2]}=\nu_{\{1,2\}}=(-1,-1,3,3), \ \ \ \ \nu_{\sigma^{0}[3]}=\nu_{\{1,2,3\}}=(0,0,0,4),\]
\[\nu_{\sigma^{1}[1]}=\nu_{\{2\}}=(2,-2,2,2), \ \ \ \ \nu_{\sigma^{1}[2]}=\nu_{\{2,3\}}=(3,-1,-1,3), \ \ \ \ \nu_{\sigma^{1}[3]}=\nu_{\{2,3,4\}}=(4,0,0,0),\]
\[\nu_{\sigma^{2}[1]}=\nu_{\{3\}}=(2,2,-2,2), \ \ \ \ \nu_{\sigma^{2}[2]}=\nu_{\{3,4\}}=(3,3,-1,-1), \ \ \ \ \nu_{\sigma^{2}[3]}=\nu_{\{1,3,4\}}=(0,4,0,0),\]
\[\nu_{\sigma^{3}[1]}=\nu_{\{4\}}=(2,2,2,-2), \ \ \ \ \nu_{\sigma^{3}[2]}=\nu_{\{1,4\}}=(-1,3,3,-1), \ \ \ \ \nu_{\sigma^{3}[3]}=\nu_{\{1,2,4\}}=(0,0,0,4).\]

If $0\neq a\in \mathbb{R}^{n},$ then $H_{a}$ will denote the
hyperplane of $\mathbb{R}^{n}$ through the origin with normal vector
a, that is,\[H_{a}=\{x\in \mathbb{R}^{n}\ | \ <x,a>=0\},\] where
$<,>$ is the usual inner product in $\mathbb{R}^{n}.$ The two closed
halfspaces bounded by $H_{a}$ are: \[H^{+}_{a}=\{x\in
\mathbb{R}^{n}\ | \ <x,a> \geq 0\} \ and \  H^{-}_{a}=\{x\in
\mathbb{R}^{n}\ | \ <x,a> \leq 0\}.\]

We will denote by $H_{\sigma^{k}[i]}$ the hyperplane of
$\mathbb{R}^{n}$ through the origin with normal vector
$\nu_{\sigma^{k}[i]}$, that is,\[H_{\nu_{\sigma^{k}[i]}}=\{x\in
\mathbb{R}^{n}\ | \ <x,\nu_{\sigma^{k}[i]}>=0\},\] for all $1\leq i
\leq n-1$ and $0\leq k \leq n-1.$

Recall that a $polyhedral\ cone\ Q \subset \mathbb{R}^{n}$ is the
intersection of a finite number of closed subspaces of the form
$H^{+}_{a}.$ If $A=\{\gamma_{1},\ldots, \ \gamma_{r}\}$ is a
finite set of points in $\mathbb{R}^{n}$ the $cone$ generated by
$A$, denoted by $\mathbb{R}_{+}{A},$  is defined as
\[\mathbb{R}_{+}{A}=\{\sum_{i=1}^{r}a_{i}\gamma_{i}\ | \ a_{i}\in
\mathbb{R}_{+},\ with \ 1\leq i \leq n\}.\]

An important fact is that $Q$ is a polyhedral cone in
$\mathbb{R}^{n}$ if and only if there exists a finite set $A\subset
\mathbb{R}^{n}$ such that $Q=\mathbb{R}_{+}{A},$ see ({\cite{BG}
\cite{W},Theorem 4.1.1.}).

Next we give some important definitions and results.(see \cite{B}, \cite{BH}, \cite{BG}, \cite{MS}, \cite{V}.)

\begin{definition}
A proper face of a polyhedral cone is a subset $F\subset \ Q$ such
that there is a supporting hyperplane $H_{a}$ satisfying:

$1)$ $F=Q\cap H_{a}\neq \emptyset.$

$2)$ $Q\nsubseteq H_{a}$ and $Q\subset H^{+}_{a}.$
\end{definition}

\begin{definition}
A cone $C$ is a pointed if $0$ is a face of $C.$ \ Equivalently we
can require that $x\in C$ \ and $-x\in C$ \ $\Rightarrow$ \ $x=0.$
\end{definition}

\begin{definition}
The 1-dimensional faces of a pointed cone are called $extremal \
rays.$
\end{definition}

\begin{definition}
If a polyhedral cone $Q$ is written as \[Q=H^{+}_{a_{1}}\cap \ldots
\cap H^{+}_{a_{r}}\] such that no one of the $H^{+}_{a_{i}}$ can be
omitted, then we say that this is an irreducible representation of
$Q.$
\end{definition}

\begin{definition}
A proper face $F$ of a polyhedral cone $Q\subset \ \mathbb{R}^{n}$
is called a $facet$ of $Q$ if $dim(F)=dim(Q)-1.$
\end{definition}

\begin{definition}
Let $Q$ be a polyhedral cone in $\mathbb{R}^{n}$ \ with $dim \ Q=n$
\ and such that $Q\neq\mathbb{R}^{n}.$ \ Let \[Q=H^{+}_{a_{1}}\cap
\ldots \cap H^{+}_{a_{r}}\] be the irreducible representation of
$Q.$ \ If $a_{i}=(a_{i1},\ldots,a_{in}),$ \ then we call
\[H_{a_{i}}(x):=a_{i1}x_{1}+\ldots + a_{in}x_{n}=0,\] \ $i\in [r],$ \ the
equations of the cone $Q.$
\end{definition}
The following result gives us the description of the relative interior of a polyhedral cone when we know the irreducible representation of it.

\begin{theorem}
Let $Q\subset \mathbb{R}^{n},$ \ $Q\neq \mathbb{R}^{n}$ \ be a polyhedral cone with $dim(Q)=n$ \ and let \[(*) \ Q=H^{+}_{a_{1}}\cap \ldots \cap H^{+}_{a_{n}}\] \ be a irreducible representation of $Q$ \ with $H^{+}_{a_{1}}, \ldots , H^{+}_{a_{n}}$ \ distinct, where $a_{i} \in \mathbb{R}^{n} \setminus \{0\}$ \ for all $i.$ \ Set $F_{i}=Q \cap H_{a_{i}}$ \ for $i \in [r].$ \ Then :\\
$a)$ \ $ri(Q)=\{x \in \mathbb{R}^{n} \ | \ <x, a_{1}> \ > 0, \ldots ,<x, a_{r}> \ > 0 \},$ \ where $ri(Q)$ \ is the relative interior of $Q,$ \ which in this case is just the interior.\\
$b)$ \ Each facet $F$ \ of $Q$ \ is of the form $F=F_{i}$ \ for some $i.$\\
$c)$ \ Each $F_{i}$ \ is a facet of $Q$ \ if and only if $(*)$ \ is irreducible.
\end{theorem}
\begin{proof}
$See \ \cite{B} \ Theorem  \ 8.2.15, \ Theorem  \ 3.2.1 .$
\end{proof}

\begin{theorem}{\bf(Danilov, Stanley)}
Let $R=K[x_{1},\ldots,x_{n}]$ be a polynomial ring over a field $K$ and $F$ a finite set
of monomials in $R.$
If $K[F]$ is normal, then the canonical module $\omega_{K[F]}$ of $K[F],$ with respect to
standard grading, can be expressed as an ideal of $K[F]$ generated by monomials
\[\omega_{K[F]}=(\{x^{a}| \ a\in {\mathbb{N}}A\cap ri({\mathbb{R_{+}}}A)\}),\] where $A=log(F)$ and
$ri({\mathbb{R_{+}}}A)$ denotes the relative interior of
${\mathbb{R_{+}}}A.$
\end{theorem}

The formula above represents the canonical module of $K[F]$ \ as an ideal of $K[F]$ \ generated by monomials. For a comprehensive treatment of the $Danilov-Stanley$ \ formula see $\cite{BH}, \ \cite{MS} \ \cite{V} . $\\

\section{Polymatroids}
\[\]

Let $K$ be a infinite field, $n$ and $m$ be positive integers,
$[n]=\{1, 2, \ldots , n\}$. A nonempty finite set $B$ of ${\bf
N}^{n}$ is the base set of a discrete polymatroid
${\bf{\mathcal{P}}}$ if for every $u=(u_{1}, u_{2}, \ldots,
u_{n})$, $v=(v_{1}, v_{2}, \ldots, v_{n})$ $\in B$ one has $u_{1}
+ u_{2} + \ldots + u_{n}=v_{1} + v_{2} + \ldots + v_{n}$ and for
all $i$ such that $u_{i}>v_{i}$ there exists $j$ such that
$u_{j}<v_{j}$ and $u+e_{j}-e_{i} \in B $, where $e_{k}$ denotes
the $k^{th}$ vector of the standard basis of ${\bf N}^{n}$. The
notion of discrete polymatroid is a generalization of the
classical notion of matroid, see \cite{E}  \cite{O} \cite{HH}
\cite{W1}. Associated with the base $B$ of a discret polymatroid
${\bf{\mathcal{P}}}$ one has a $K-$algebra $K[ B ]$, called the
base ring of ${\bf{\mathcal{P}}}$, defined to be the
$K-$subalgebra of the polynomial ring in $n$ indeterminates
$K[x_{1}, x_{2}, \ldots , x_{n}]$ generated by the monomials
$x^{u}$ with $u \in B$. From \cite{HH} the algebra $K[ B ]$ is
known to be normal and hence Cohen-Macaulay.

If $A_{i}$ are some non-empty subsets of $[n]$ for $1\leq i\leq
m$, ${\bf{\mathcal{A}}}=\{A_{1},\ldots,A_{m}\}$, then the
set of the vectors $\sum_{k=1}^{m} e_{i_{k}}$ with $i_{k} \in
A_{k},$ is the base of a polymatroid, called transversal
polymatroid presented by ${\bf{\mathcal{A}}}.$ The base ring of a
transversal polymatroid presented by ${\bf{\mathcal{A}}}$ denoted
by $K[{\bf{\mathcal{A}}}]$ is the ring :
\[K[{\bf{\mathcal{A}}}]:=K[x_{i_{1}}\cdot\cdot\cdot
x_{i_{m}}:i_{j}\in A_{j},1\leq j\leq m].\]

\section{Cones of dimension $n$ with $n+1$ facets.}
\[\]

\begin{lemma}
Let $1\leq i\leq n-2,$ $A:=\{
log(x_{j_{1}}\cdot\cdot\cdot x_{j_{n}}) \ | \
j_{k}\in A_{k},\ for \ all \ 1\leq k \leq n
\}\subset \mathbb{N}^{n}$ the exponent set of generators of
$K-$algebra $K[{\bf{\mathcal{A}}}],$ where
${\bf{\mathcal{A}}}=\{A_{1}=[n],\ldots,A_{i}=[n],
A_{i+1}=[n]\setminus [i],\ldots,A_{n-1}=[n]\setminus
[i],A_{n}=[n]\}$. Then the cone generated by $A$ has the
irreducible representation:
\[\mathbb{R}_{+}A= \bigcap_{a\in N}H^{+}_{a},\] where
$N=\{\nu_{\sigma^{0}[i]}, \nu_{\sigma^{k}[n-1]} \ | \ 0\leq k \leq
n-1\}.$
\end {lemma}

\begin{proof}
We denote by 
$
J_{k}=\left \{ \begin{array}[pos]{cccccccccccccccc}
(i+1) \ e_{k}+ (n-i-1) \ e_{i+1} \ , \ \ if \ \ 1\leq k \leq i \ \ \ \ \ \ \\

(i+1) \ e_{1}+(n-i-1) \ e_{k} \ , \ \ \ \ \ if \ i+2\leq k \leq n \ \
\end{array} \right .
$ and by $J=n \ e_{n}.$ \
Since $A_{t}=[n]$ \ for any $t \in \{1,\ldots, i\}\cup \{n\}$ \ and $A_{r}=[n] \setminus [i]$ \ for any $r \in \{i+1,\ldots, n-1\}$ it is easy to see that for any $k \in \{1,\ldots, i\}$ \ and  $r \in \{i+2,\ldots, n\}$ \ the set of monomials $x^{i+1}_{k} \ x^{n-i-1}_{i+1},$ \ $x^{i+1}_{1} \ x^{n-i-1}_{r},$ \ $x^{n}_{n}$ \ are a subset of the generators of $K-$algebra $K[{\bf{\mathcal{A}}}].$ \ Thus the set \[\{J_{1}, \ldots, J_{i}, J_{i+2}, \ldots, J_{n}, J\}\subset A.\]
If we denote by $C$ \ the matrix with the rows the coordinates of the $\{J_{1}, \ldots, J_{i}, J_{i+2}, \ldots, J_{n}, J\},$ \ then by a simple computation we get $|det\ (C)|=n \ (i+1)^{i} \ (n-i-1)^{n-i-1}$ \ for any $1\leq i\leq n-2$ \ and $1\leq j\leq n-1.$ \ Thus, we get that the set \[\{J_{1}, \ldots, J_{i}, J_{i+2}, \ldots, J_{n}, J\}\] is lineary independent and it follows that $dim \ \mathbb{R}_{+}A = n.$ \ Since $\{J_{1}, \ldots, J_{i}, J_{i+2}, \ldots, J_{n}\}$ is linearly independent and lie on the
hyperplane $H_{\sigma^{0}[i]}$ we have that $dim(H_{\sigma^{0}[i]}\cap \mathbb{R}_{+}A)=n-1.$\\Now we will prove that $\mathbb{R}_{+}A \subset H_{a}^{+}$ for all $a \in N.$ It is enough to show that for all vectors $P \in A,$ \ $<P,a> \geq 0$ for all $a \in N.$ Since $\nu_{\sigma^{k}[n-1]}=n \ e_{[n]\setminus
\sigma^{k}[n-1]},$ where $\{e_{i}\}_{1 \leq i \leq n}$ \ is  the
canonical base of $\mathbb{R}^{n},$\ we get that $<P,\nu_{\sigma^{k}[n-1]}> \geq 0.$ \ Let $P \in A,$ \ $P=log(x_{j_{1}}\cdot\cdot\cdot x_{j_{i}}x_{j_{i+1}}\cdot\cdot\cdot x_{j_{n-1}}x_{j_{n}})$ and let $t$ to be the number of $j_{k_{s}},$ \ such that \ $1 \leq k_{s} \leq i$ \ and $j_{k_{s}} \in [i].$ \ Thus $1 \leq t \leq i.$ \ Now we have only two cases to consider:\\ $1)$ \ If $j_{n}\in [i],$ then $<P,\nu_{\sigma^{0}[i]}>= -t(n-i-1)+(i-t)(i+1)+(n-i-1)(i+1)-(n-i-1)=n(i-t)\geq 0.$ \\ $2)$ \ If $j_{n}\in [n]\setminus [i],$ then $<P,\nu_{\sigma^{0}[i]}>=-t(n-i-1)+(i-t)(i+1)+(n-i-1)(i+1)+(i+1)=n(i-t+1)>0.$ \\ Thus \[\mathbb{R}_{+}A\subseteq \bigcap_{a\in N}H^{+}_{a}.\]Now we will prove the converse inclusion: $\mathbb{R}_{+}A\supseteq \bigcap_{a\in N}H^{+}_{a}.$\\ It is clearly enough to prove that the extremal rays of the cone $\bigcap_{a\in N}H^{+}_{a}$ \ are in
${\mathbb{R}_{+}}A.$ \ Any extremal ray of the cone  $\bigcap_{a\in N}H^{+}_{a}$ \ can be written  as the intersection of $n-1$ hyperplanes $H_{a}$ \ , with $a\in N.$ \ There are two possibilities to obtain extremal rays by intersection of $n-1$ hyperplanes.\\ $First \ case.$\\ Let \ $1\leq i_{1}< \ldots < i_{n-1}\leq n$ \ be a sequence and $\{t\}=[n]\setminus \{i_{1},\ldots,i_{n-1}\}.$ \ The system of equations:
$ \ (*) \ \begin{cases}
z_{i_{1}}=0,\\
\vdots \\
z_{i_{n-1}}=0
\end{cases}$ admits the solution $x \in \mathbb{Z}_{+}^{n},$ \ $
x=\left( \begin{array}[pos]{c}
    x_{1}\\
    \vdots\\
    x_{n}
\end{array} \right)
$ with $ | \ x \ |=n,$ \ $x_{k}=n\cdot\delta_{kt}$ \ for all $1\leq k \leq n,$ where $\delta_{kt}$ is Kronecker symbol.\\There are two possibilities:\\ $1)$ \ If $1\leq t \leq i,$ \ then $ H_{\sigma^{0}[i]}(x) \ < \ 0 \ and \ thus \ x \ \notin \bigcap_{a\in N}H^{+}_{a}.$\\
$2)$ \ If $i+1\leq t \leq n,$ \ then $H_{\sigma^{0}[i]}(x) \ > \ 0$ \ and thus $x\in \bigcap_{a\in N}H^{+}_{a}$ and  is an extremal ray.\\ Thus, there exist $n-i$ \ sequences \ $1\leq i_{1}< \ldots < i_{n-1}\leq n$ \ such that the system of equations $(*)$ \ has a solution $x \in \mathbb{Z}_{+}^{n}$ \ with $ | \ x \ |=n$ \ and $H_{\sigma^{0}[i]}(x) \ > \ 0.$\\ The extremal rays are: $\{ne_{k} \ | \ i+1 \leq k \leq n\}.$\\
$Second \ case.$\\ Let \ $1\leq i_{1}< \ldots < i_{n-2}\leq n$ \ be a sequence and $\{j, \ k\}=[n]\setminus \{i_{1},\ldots,i_{n-2}\},$ \ with $j < k$ \ and 
$ \ (**) \ \begin{cases}
z_{i_{1}}=0,\\
\vdots \\
z_{i_{n-2}}=0,\\
-(n-i-1)z_{1}-\ldots-(n-i-1)z_{i}+(i+1)z_{i+1}+\ldots+(i+1)z_{n}=0
\end{cases}$\\
be the system of linear equations associated to this sequence.\\
There are two possibilities:\\
$1)$ \ If $1\leq j \leq i$ and $i+1\leq k \leq n,$ then the system of equations $(**)$ admits the solution
$
x=\left( \begin{array}[pos]{c}
    x_{1}\\
    \vdots\\
    x_{n}
\end{array} \right) \in \mathbb{Z}_{+}^{n},$ \ with $ | \ x \ |=n,$ \ with $x_{t}=(i+1)\delta_{jt}+(n-i-1)\delta_{kt}$ \ for all $1\leq t \leq n.$\\
$2)$ If $1\leq j,k \leq i$ \ or $i+1\leq j,k \leq n,$ \ then there exist no solution $x\in \mathbb{Z}_{+}^{n}$ \ with $| \ x \ | \ = \ n$ for the system of equations $(**)$ \ because otherwise $H_{\sigma^{0}[i]}(x) \ > \ 0$ \ or $H_{\sigma^{0}[i]}(x) < 0.$\\
Thus, there exist $i(n-i)$ \ sequences \ $1\leq i_{1}< \ldots < i_{n-2}\leq n$ such that the system of equations $(**)$ \ has a solution $x \in \mathbb{Z}_{+}^{n}$ \ with $ | \ x \ |=n$ \ and the extremal rays are: $\{(i+1)e_{j}+(n-i-1)e_{k} \ | \ 1 \leq j \leq i \ and \ i+1 \leq k \leq n\}.$\\In conclusion, there exist $(i+1)(n-i)$ \ extremal rays of the cone ${\bigcap}_{a\in N}H_{a}^{+}$:\[R:=\{n e_{k} \ | \ i+1 \leq k \leq n\}\cup\{(i+1)e_{j}+(n-i-1)e_{k} \ | \ 1 \leq j \leq i \ and \ i+1 \leq k \leq n\}.\]Since $R\subset A$ \ we have $\mathbb{R}_{+}A= \bigcap_{a\in N}H^{+}_{a}.$\\
It is easy to see that the representation is
irreducible because if we delete, for some $k,$ \ the hyperplane with
the normal $\nu_{\sigma^{k}[n-1]}$, then a coordinate
of a $log(x_{j_{1}}\cdot\cdot\cdot x_{j_{i}}x_{j_{i+1}}\cdot\cdot\cdot x_{j_{n-1}}x_{j_{n}})$
could be negative, which is impossible; and if we delete the
hyperplane with the normal $\nu_{\sigma^{0}[i]},$ then the cone
${\mathbb{R_{+}}}A$ would be generated by $A=\{
log(x_{j_{1}}\cdot\cdot\cdot x_{j_{n}}) \ | \ j_{k}\in [n], \ for \ all \ 1\leq
k \leq n \}$ which is impossible. \ Thus the representation
${\mathbb{R_{+}}}A= \bigcap_{a\in N}H^{+}_{a}$ is irreducible.
\end{proof}

\begin{lemma}
Let $1\leq i\leq n-2, \ 1\leq t\leq n-1,$ $A:=\{ log(x_{j_{1}}\cdot\cdot\cdot x_{j_{n}}) \ | \
j_{\sigma^{t}{(k)}}\in A_{\sigma^{t}{(k)}}, 1\leq k \leq
n\}\subset \mathbb{N}^{n}$ the exponent set of generators of
$K-$algebra  $K[{\bf{\mathcal{A}}}],$ where ${\bf{\mathcal{A}}}=\{
A_{\sigma^{t}{(k)}} \ | \ A_{\sigma^{t}{(k)}}=[n],\ for \ 1\leq
k\leq i \ and \ A_{\sigma^{t}{(k)}}=[n]\setminus \sigma^{t}[i], \ for \ i+1\leq k\leq
n-1,\ A_{\sigma^{t}{(n)}}=[n]\}$. Then the cone generated by $A$
has the irreducible representation:
\[{\mathbb{R_{+}}}A= \bigcap_{a\in N}H^{+}_{a},\] where $N=\{\nu_{\sigma^{t}[i]},\ \nu_{\sigma^{k}[n-1]} \ | \ 0\leq k \leq
n-1\}.$
\end {lemma}

\begin{proof}
The proof goes as in Lemma 4.1. since the algebras from Lemmas 4.1. and 4.2. are
isomorphic.
\end{proof}

\section{The a-invariant and the canonical module}

\begin{lemma}
The $K-$ algebra
$K[{\bf{\mathcal{A}}}],$ where ${\bf{\mathcal{A}}}=\{
A_{\sigma^{t}{(k)}} \ | \ A_{\sigma^{t}{(k)}}=[n],\ for \
1\leq k\leq i, \ and \ A_{\sigma^{t}{(k)}}=[n]\setminus
\sigma^{t}[i], \ for \ i+1\leq k\leq n-1,\ A_{\sigma^{t}{(n)}}=[n]\},$ is Gorenstein ring
for all $0\leq t\leq n-1$ and $1\leq i \leq n-2$.
\end{lemma}

\begin{proof}
Since the algebras from Lemmas 4.1 and 4.2 are
isomorphic it is enough to prove the case $t=0.$

We will show that the canonical module $\omega_{K[{\bf{\mathcal{A}}}]}$ is generated by $(x_{1}\cdot\cdot\cdot x_{n})K[{\bf{\mathcal{A}}}].$ Since $K-$ algebra $K[{\bf{\mathcal{A}}}]$ is normal, using the $Danilov-Stanley$ theorem we get that the canonical module $\omega_{K[{\bf{\mathcal{A}}}]}$ is \[\omega_{K[{\bf{\mathcal{A}}}]}=\{x^{\alpha} \ | \ \alpha \in \mathbb{N}A \cap ri(\mathbb{R}_{+}A)\}.\]\\
Let $d$ be the greatest common divizor of $n$ and $i+1,$ $d=gcd(n, \ i+1),$ then the equation of the facet $H_{\nu_{\sigma^{0}[i]}}$ is 
\[H_{\nu_{\sigma^{0}[i]}}: \ -\frac{(n-i-1)}{d}\sum_{k=1}^{i}x_{k}+ \frac{(i+1)}{d}\sum_{k=i+1}^{n}x_{k}=0.\] The  relative interior of the cone $\mathbb{R}_{+}A$ is:  \[ri(\mathbb{R}_{+}A)=\{x \in \mathbb{R}^{n} \ | \ x_{k} \ > \ 0, \ for \ all \ 1\leq k \leq n, \ -\frac{(n-i-1)}{d}\sum_{k=1}^{i}x_{k}+ \frac{(i+1)}{d}\sum_{k=i+1}^{n}x_{k} \ > \ 0 \}.\]  We will show that $\mathbb{N}A \cap ri(\mathbb{R}_{+}A)= \ (1,\ldots, 1) \ + \ (\mathbb{N}A \cap \mathbb{R}_{+}A).$\\ It is clear that $ri(\mathbb{R}_{+}A) \ \supset \ (1,\ldots, 1) \ + \mathbb{R}_{+}A.$\\ If $(\alpha_{1}, \alpha_{2}, \ldots, \alpha_{n})\in \mathbb{N}A \cap ri(\mathbb{R}_{+}A),$ then $\alpha_{k} \geq 1, for \ all \ 1\leq k\leq n$ and \[-\frac{(n-i-1)}{d}\sum_{k=1}^{i}\alpha_{k}+ \frac{(i+1)}{d}\sum_{k=i+1}^{n}\alpha_{k} \ \geq 1 \ and \ \sum_{k=1}^{n}\alpha_{k}=t \ n \ for \ some \ t\geq 1.\] 

We claim that there exist $(\beta_{1}, \beta_{2}, \ldots, \beta_{n})\in \mathbb{N}A \cap \mathbb{R}_{+}A$ such that $(\alpha_{1}, \alpha_{2}, \ldots, \alpha_{n})=(\beta_{1}+1, \beta_{2}+1, \ldots, \beta_{n}+1).$
Let $\beta_{k}=\alpha_{k}-1$ for all $1\leq k \leq n.$ It is clear that $\beta_{k}\geq 0$ and \[-\frac{(n-i-1)}{d}\sum_{k=1}^{i}\beta_{k}+ \frac{(i+1)}{d}\sum_{k=i+1}^{n}\beta_{k}=-\frac{(n-i-1)}{d}\sum_{k=1}^{i}\alpha_{k}+ \frac{(i+1)}{d}\sum_{k=i+1}^{n}\alpha_{k}-\frac{n}{d}.\] 
\[If -\frac{(n-i-1)}{d}\sum_{k=1}^{i}\alpha_{k}+ \frac{(i+1)}{d}\sum_{k=i+1}^{n}\alpha_{k}=j \ with \ 1\leq j \leq \frac{n}{d}-1, \ then \ we \ will \ get \ a \ contadiction.\] Indeed, since $n$ divides $\sum_{k=1}^{n}\alpha_{k},$ it follows $\frac{n}{d}$ divides $j$ which is false.\\ So we have  \[-\frac{(n-i-1)}{d}\sum_{k=1}^{i}\beta_{k}+ \frac{(i+1)}{d}\sum_{k=i+1}^{n}\beta_{k}=-\frac{(n-i-1)}{d}\sum_{k=1}^{i}\alpha_{k}+ \frac{(i+1)}{d}\sum_{k=i+1}^{n}\alpha_{k}-\frac{n}{d}\geq 0.\] Thus $(\beta_{1}, \beta_{2}, \ldots, \beta_{n})\in \mathbb{N}A \cap \mathbb{R}_{+}A$ and $(\alpha_{1}, \alpha_{2}, \ldots, \alpha_{n})\in \mathbb{N}A \cap ri(\mathbb{R}_{+}A)$.\\ Since $\mathbb{N}A \cap ri(\mathbb{R}_{+}A)= \ (1,\ldots, 1) \ + \ (\mathbb{N}A \cap \mathbb{R}_{+}A),$ we get that $\omega_{K[{\bf{\mathcal{A}}}]}= \ (x_{1}\cdot\cdot\cdot x_{n})K[{\bf{\mathcal{A}}}].$
\end{proof}

Let $S$ be a standard graded $K-algebra$ over a field $K$. Recall that the $a-invariant$ of $S,$
denoted $a(S),$ is the degree as a rational function of the Hilbert series of $S,$ see for instance
$(\cite{V}, \ p. \ 99).$ If $S$ is $Cohen-Macaulay$ and $\omega_{S}$ is the canonical module of $S,$ then \[a(S)= \ - \ min \ \{i\ | \ (\omega_{S})_{i} \neq 0\},\] see (\cite{BH}, \ p. \ 141) and (\cite{V}, \ Proposition 4.2.3).
In our situation $S = K[{\bf{\mathcal{A}}}]$ is normal \cite{HH} and
consequently $Cohen-Macaulay,$ thus this formula applies. As consequence of $Lemma \ 5.1.$ we have the following:
\begin{corollary}
The $a-invariant$ of $K[{\bf{\mathcal{A}}}]$ is $a(K[{\bf{\mathcal{A}}}])= \ -1.$
\end{corollary}
\begin{proof}
Let $\{x^{\alpha_{1}}, \ldots, x^{\alpha_{q}}\}$  the  generators  of  $K-algebra$ $K[{\bf{\mathcal{A}}}].$ \  $K[{\bf{\mathcal{A}}}]$ is standard graded algebra with the grading
\[K[{\bf{\mathcal{A}}}]_{i}=\sum_{|c|=i}K(x^{\alpha_{1}})^{c_{1}}\cdot \cdot \cdot (x^{\alpha_{q}})^{c_{q}}, \ where \ |c|=c_{1}+\ldots+c_{q}. \] Since $\omega_{K[{\bf{\mathcal{A}}}]}= \ (x_{1}\cdot\cdot\cdot x_{n})K[{\bf{\mathcal{A}}}]$ it follows that $ min \ \{i\ | \ (\omega_{K[{\bf{\mathcal{A}}}]})_{i} \neq 0\}=1,$ thus $a(K[{\bf{\mathcal{A}}}])= \ -1.$
\end{proof}
\section{Ehrhart function}
We consider  a fixed set of distinct monomials
$F=\{x^{\alpha_{1}},\ldots,x^{\alpha_{r}}\}$ \ in a polynomial ring
$R=K[x_{1},\ldots, x_{n}]$ \ over a field $K.$ \\ Let
\[\mathcal{P}=conv(log(F))\] \ be the convex hull of the set $log(F)=\{\alpha_{1},\ldots,
\alpha_{r}\}.$ \\ The $normalized \ Ehrhart \ ring$ \ of $\mathcal{P}$ is the
graded algebra \[A_{\mathcal{P}}=\bigoplus^{\infty}_{j=0}(A_{\mathcal{P}})_{j} \subset
R[T]\] where the $j-th$ \ component  is given by
\[(A_{\mathcal{P}})_{j}=\sum_{\alpha \in \mathbb{Z} \ log(F)\cap \ 
j\mathcal{P}} \ K \ x^{\alpha} \ T^{j}.\]
The $normalized \ Ehrhart \ function$ of $\mathcal{P}$ is defined as
\[E_{\mathcal{P}}(j)= dim_{K}(A_{\mathcal{P}})_{j}=|\ \mathbb{Z} \ {log(F)}\cap j\mathcal{P} \ |.\] From
{\cite{V}, Proposition 7.2.39  and  Corollary 7.2.45} we have the following important result:
\begin{theorem}
If $K[F]$ is a standard graded subalgebra of $R$ and $h$ \ is the
$Hilbert \ function$ of $K[F],$ \ then:\\
$a)$ \ $h(j)\leq E_{\mathcal{P}}(j)$ \ for all $j\geq 0,$ \ and\\
$b)$ \ $h(j)= E_{\mathcal{P}}(j)$ \ for all $j\geq 0$ \ if and only if $K[F]$
is normal.
\end{theorem}

In this section we will compute the Hilbert function and the Hilbert series for $K-$ algebra $K[{\bf{\mathcal{A}}}],$ where $\bf{\mathcal{A}}$ satisfied the hypothesis of $Lemma \ 4.1.$
\begin{proposition}
In the hypothesis of $Lemma \ 4.1.$, the Hilbert function of $K-$ algebra $K[{\bf{\mathcal{A}}}]$ is :\[h(t)=\sum_{k=0}^{(i+1)t}\binom{k+i-1}{k}\binom{nt-k+n-i-1}{nt-k}.\]
\end{proposition}

\begin{proof}
From {\cite{HH}} we know that the $K-$ algebra $K[{\bf{\mathcal{A}}}]$ is normal. Thus, to compute the Hilbert function of $K[{\bf{\mathcal{A}}}]$ it is equivalent to compute the Ehrhart function of ${\mathcal{P}},$ where ${\mathcal{P}} \ = \ conv(A).$  \\ It is clear enough that $\mathcal{P}$ is the intersection of the cone $\mathbb{R}_{+}A$ with the hyperplane $x_{1}+ \ldots +x_{n}=n,$ thus \[\mathcal{P}=\{ \alpha\in \mathbb{R}^n \ | \ \alpha_{k}\geq 0 \ for \ any \ k\in [n], \ 0\leq \alpha_{1}+\ldots + \alpha_{i}\leq i+1 \ and \ \alpha_{1}+\ldots + \alpha_{n} = n \}\] and  
\[t \ \mathcal{P}=\{ \alpha\in \mathbb{R}^n \ | \ \alpha_{k}\geq 0 \ for \ any \ k\in [n], \ 0\leq \alpha_{1}+\ldots + \alpha_{i}\leq (i+1) \ t \ and \ \alpha_{1}+\ldots + \alpha_{n} = n \ t \}.\]

Since for any $0\leq k \leq (i+1) \ t$ the equation $ \alpha_{1}+\ldots + \alpha_{i} = k$ has $\binom{k+i-1}{k}$ nonnegative integer solutions and the equation $\alpha_{i+1}+\ldots + \alpha_{n} = n \ t -k$ has $\binom{nt-k+n-i-1}{nt-k}$ nonnegative integer solutions, we get that \[E_{\mathcal{P}}(t)=|\ \mathbb{Z} \ A \ \cap \ t \ \mathcal{P} \ | = \sum_{k=0}^{(i+1)t}\binom{k+i-1}{k}\binom{nt-k+n-i-1}{nt-k}.\]
\end{proof}

\begin{corollary}
The Hilbert series of $K-$ algebra $K[{\bf{\mathcal{A}}}],$ where $\bf{\mathcal{A}}$ satisfied the hypothesis of $Lemma \ 4.1.$ is :\[H_{K[{\bf{\mathcal{A}}}]}(t) \ = \ \frac{1+h_{1} \ t + \ldots + h_{n-1} \ t^{n-1} }{(1-t)^{n}},\] where 
\[h_{j}=\sum_{s=0}^{j}(-1)^{s} \ h(j-s) \ \binom{n}{s},\]
$h(s)$ is the Hilbert function of $K[{\bf{\mathcal{A}}}]$.
\end{corollary}

\begin{proof}
Since the $a-invariant$ of $K[{\bf{\mathcal{A}}}]$ is $a(K[{\bf{\mathcal{A}}}])= \ -1,$ it follows that to compute the Hilbert series of $K[{\bf{\mathcal{A}}}]$ is necessary to know the first $n$ values of the Hilbert function of $K[{\bf{\mathcal{A}}}]$, $h(i)$ for $0\leq i \leq n-1.$
Since $dim(K[{\bf{\mathcal{A}}}])=n,$ apllying $n$ times the $difference \ operator \ \Delta$ \ $ ({see \ \cite{BH}})$  on the Hilbert function of $K[{\bf{\mathcal{A}}}]$ we get the conclusion.\\
Let $\Delta^{0}(h)_{j}:=h(j)$ for any
$0\leq j \leq n-1.$\\ For $k\geq 1$ let $\Delta^{k}(h)_{0}:=1$
and $\Delta^{k}(h)_{j}:=\Delta^{k-1}(h)_{j}-\Delta^{k-1}(h)_{j-1}$
for any $1\leq j \leq n-1.$\\
We claim that:
\[\Delta^{k}(h)_{j}=\sum_{s=0}^{k}(-1)^sh(j-s)\binom{k}{s}\]
for any $k\geq 1$ and $0\leq j \leq n-1.$\\
We proceed by induction on $k.$\\
If $k=1,$ then
\[
\Delta^{1}(h)_{j}=\Delta^{0}(h)_{j}-\Delta^{0}(h)_{j-1}=h(j)-h(j-1)
=\sum_{s=0}^{1}(-1)^{s}h(j-s)\binom{1}{s}
\]
for any $1\leq j \leq n-1.$\\
If $k>1,$ then
\[\Delta^{k}(h)_{j}=\Delta^{k-1}(h)_{j}-\Delta^{k-1}(h)_{j-1}=\sum_{s=0}^{k-1}(-1)^{s}h(j-s)
\binom{k-1}{s}-\sum_{s=0}^{k-1}(-1)^{s} h(j-1-s)
\binom{k-1}{s}=\]
\[=h(j)\binom{k-1}{0}+\sum_{s=1}^{k-1}(-1)^{s}h(j-s)
\binom{k-1}{s}-\sum_{s=0}^{k-2}(-1)^{s} h(j-1-s)
\binom{k-1}{s}+(-1)^kh(j-k)\binom{k-1}{k-1}=\]
\[=h(j)+\sum_{s=1}^{k-1}(-1)^{s}h(j-s)\left[\binom{k-1}{s}+\binom{k-1}{s-1}\right]+(-1)^kh(j-k)\binom{k-1}{k-1}=\ \ \ \ \ \ \ \ \ \ \ \ \ \ \ \ \ \ \ \ \ \ \ \ \ \ \ \ \]
\[=h(j)+\sum_{s=1}^{k-1}(-1)^{s}h(j-s)\binom{k}{s}+(-1)^kh(j-k)\binom{k-1}{k-1}=\sum_{s=0}^{k}(-1)^{s}h(j-s)\binom{k}{s}.\ \ \ \ \ \ \ \ \ \ \ \ \ \ \ \ \ \ \ \ \ \ \ \ \]
Thus, if $k=n$ it follows that: 
\[h_{j}=\Delta^{n}(h)_{j}=\sum_{s=0}^{n}(-1)^sh(j-s)\binom{n}{s}=\sum_{s=0}^{j}(-1)^sh(j-s)\binom{n}{s}\]
for any $1\leq j\leq n-1.$

\end{proof}

\vspace{2mm} \noindent {\footnotesize
\begin{minipage}[b]{10cm}
Alin \c{S}tefan, Assistant Professor\\
"Petroleum and Gas" University of Ploie\c{s}ti\\
Ploie\c{s}ti, Romania\\
E-mail:nastefan@upg-ploiesti.ro
\end{minipage}}
\end{document}